\documentclass[11pt]{amsart}
\usepackage{geometry}                
\usepackage{graphicx}
\usepackage{amssymb}
\usepackage{epstopdf}
\usepackage{amsmath}
\usepackage{color}
\usepackage{hyperref}
 \usepackage{tikz}
 
\usepackage[all]{xy}
\xyoption{matrix}
\xyoption{arrow}

\usetikzlibrary{arrows,decorations.pathmorphing,backgrounds,positioning,fit,petri}

 \tikzset{help lines/.style={step=#1cm,very thin, color=gray},
help lines/.default=.5} 
\tikzset{thick grid/.style={step=#1cm,thick, color=gray},
thick grid/.default=1} 

\newtheorem{thm}{Theorem}[section]
\newtheorem*{thm*}{Theorem}

\newtheorem{lem}[thm]{Lemma}
\newtheorem{cor}[thm]{Corollary}
\newtheorem{prop}[thm]{Proposition}

\theoremstyle{definition}
\newtheorem{defn}[thm]{Definition}
\newtheorem{eg}[thm]{Example}

\theoremstyle{remark}
\newtheorem{rem}[thm]{Remark}
 
\numberwithin{equation}{section}

\newcommand{\field}[1]{\mathbb{#1}}

\newcommand{\CC}{\ensuremath{{\field{C}}}}

\newcommand{\commentout}[1]{}

\title{Frieze varieties are invariant under  Coxeter mutation}

\author{Kiyoshi Igusa}
\address{Department of Mathematics, Brandeis University, Waltham, MA 02454}\email{igusa@brandeis.edu}

\author{Ralf Schiffler}\thanks{The second author was supported by the NSF grants  DMS-1254567 and  DMS-1800860 and by the University of Connecticut.}
\address{Department of Mathematics, University of Connecticut, USA}
\email{schiffler@math.uconn.edu}
\subjclass[2010]{13F60,16G20}

\begin{document}

\maketitle

\begin{abstract}
We define a generalized version of the frieze variety introduced by  {Lee, Li, Mills, Seceleanu and} the second author. The generalized frieze variety is an algebraic variety determined by an acyclic quiver and a generic specialization of cluster variables in the cluster algebra for this quiver. {The original frieze variety is obtained when this specialization is $(1,\ldots,1)$}.

The main result is that a generalized frieze variety is determined by any generic element of any component of that variety. We also show that the ``Coxeter mutation'' cyclically permutes these components. In particular, this shows that the frieze variety is invariant under the Coxeter mutation at a generic point.

The paper contains many examples which are generated using a new technique which we call an invariant Laurent polynomial. We show that a symmetry of a mutation of a quiver gives such an invariant rational function.
\end{abstract}

\tableofcontents

\section*{Introduction}

At the {ARTA VI Conference in Mexico} celebrating Jose Antonio de la Pe\~na's 60th birthday, the second author presented his paper \cite{LLMSS} defining ``frieze varieties'' of an acyclic quiver and, using a result of Jose Antonio de la Pe\~na \cite{dlP} on eigenvalues of the Coxeter matrix of the quiver, to prove the main result: The dimension of this frieze variety is equal to {0,1, or $\ge 2$  if and only if the representation type of the quiver is finite, tame, or wild, respectively. In particular, the dimension of the frieze variety is 1 if and only if $Q$ is an extended Dynkin quiver.}

After {the} talk, the two authors discussed properties and examples of frieze varieties throughout the conference. This paper is a report on these discussions.

The basic idea was to generalize the notion of frieze varieties by allowing for arbitrary generic vectors $(b_1, b_2,\ldots,b_n)$ instead of the vector $(1,1, \ldots,1)$ as the initial condition for the defining recurrence. We show that, for any generic point $a_\ast=(a_1,\ldots,a_n)$ in any component of the (generalized) frieze variety, the Coxeter mutations of $a_\ast$ \eqref{Coxeter mutation} will be contained as a dense subset of the variety. (Theorems \ref{thm: generic points of X(Q) generate X(Q)}, \ref{thm: generic points of X(Q,b) generate X(Q,b)})

We also show that the Coxeter mutation cyclically permutes the components of the  (generalized) frieze variety (Theorem \ref{thm: mu ast cyclically permutes components of X}). In nice examples, we can use this cyclic permutation to generate all polynomials which define all components of the generalized frieze variety out of a single rational function. (Proposition \ref{prop: formulas from one h})

 Finally, we also observe that a symmetry of a mutation of a quiver decreases the dimension of the frieze variety. (Proposition \ref{prop: symmetry of a mutation of Q})

The authors wish to thank the organizers of the ARTA VI conference  for a very enjoyable and productive event. We also wish to say a special ``Happy Birthday'' to Jos\'e Antonio de la Pe\~na and congratulations on his numerous achievements. This paper was also presented by the first author and referenced by the second author at the conference ``Cluster Algebras and Representation Theory'' held in Kyoto in June 2019. Key observations by Gordana Todorov before and during that conference are added in Section \ref{sec: Q+A} with details to be given in another paper. Also, Salvatore Stella and Alastair King gave us very helpful comments about the first version of this paper.


\section{Preliminaries}\label{sect 1}
{We recall the main result of \cite{LLMSS}.  Let $Q$ be a connected finite quiver without oriented cycles and label the vertices $1,2,\ldots,n$ such that if there is an arrow $i\to j$ then $i>j$.

\begin{defn}\cite{LLMSS}\label{def 1}
\begin{enumerate}
\item For every vertex $i\in Q_0$ define positive rational numbers $f_i(t)$ ($t\in\mathbb{Z}_{\ge0}$) recursively by  $f_i(0)=1$ and 
\begin{equation}\nonumber
f_i(t+1)=\frac{1+\prod_{j\to i}f_j(t) \prod_{j\leftarrow i}f_j(t+1)}{f_i(t)}.
\end{equation}
\item 
For every $t\ge 0$, define the point  $P_t=(f_1(t),\dots,f_n(t))\in\mathbb{C}^n$.
\item 
The \emph{frieze variety} $X(Q)$ of the quiver $Q$ is the Zariski closure of the set of all  points $P_t$ ($t\in\mathbb{Z}_{\ge0}$). 
\end{enumerate}
\end{defn}

The main result in \cite{LLMSS} is the following characterization of the finite--tame--wild trichotomy for acyclic quivers $Q$ in terms of its frieze variety $X(Q)$.


\begin{thm}\cite{LLMSS}\label{thm main}
Let $Q$ be an acyclic quiver.
\begin{itemize}
\item [{\rm (a)}] If $Q$ is representation finite then  the frieze variety $X(Q)$ is of dimension 0.
\item [{\rm (b)}] If $Q$ is tame then the frieze variety $X(Q)$ is of dimension 1.
\item [{\rm (c)}] If $Q$ is wild then the frieze variety $X(Q)$ is of dimension at least 2.
\end{itemize}
\end{thm}
}


\section{Definitions and main result}
Let $Q$ be as in section \ref{sect 1}. 
Let $\mathcal{A}(Q)$ be the cluster algebra of $Q$ and let $\mathbf{x}=(x_1,\ldots,x_n)$ be the initial cluster in $\mathcal{A}(Q)$. Let $\mu_k$ denote the mutation in direction $k$ and let $x_k'$ be the new cluster variable obtained by this mutation, thus $\mu_k(\mathbf{x})=(x_1,\ldots,x_{k-1},x_k',x_{k+1},\ldots,x_n)$. We define the \emph{Coxeter mutation} to be the mutation sequence 
\begin{equation}\label{Coxeter mutation}
\mu_\ast=\mu_n\circ\cdots \circ\mu_2 \circ \mu_1,
\end{equation}
where the order $1,2,\ldots,n$ of the vertices is as in section \ref{sect 1}. It is shown in \cite{LLMSS} that the point $P_t$ in Definition \ref{def 1} is equal to the specialization of the cluster $\mu_\ast^t(\mathbf{x})$ at $x_i=1$.

For an arbitrary point $a_\ast=(a_1,\cdots,a_n)\in\CC^n$ with $a_i\neq0$ and any $1\le k\le n$, let
\[
	\mu_k(a_1,\cdots,a_n)=(a_1,\cdots,a_{k-1},a_k',a_{k+1},\cdots, a_n)\in \CC^n
\]
{where $a_k'$ is obtained from the cluster variable $x_k'$ by specializing the initial cluster variables $x_i=a_i$, $i=1,2,\ldots,n$.}  For generic $a_\ast$ we will have $a_k'\neq0$, so the Coxeter mutation can be repeated.
Let $\mu_\ast(a_\ast)=\mu_n\circ\cdots \circ \mu_1(a_\ast)$ .

\smallskip 

{
We propose the following generalization of  the frieze variety.}


\begin{defn}\label{def of X(Q,a)}
 \begin{enumerate}
\item We say that $a_\ast\in \CC^n$ is a \emph{generic specialization} of  $\mathbf{x}$ if all coordinates of $\mu_\ast^t(a_\ast)$ are nonzero for all $t\ge0$.
\item 
We refer to the set of all $\mu_\ast^t(a_\ast)\in\CC^n$ for $t\ge0$ as the \emph{$\mu_\ast$-orbit} of $a_\ast$.

\item For any generic specialization $a_\ast$ of $\mathbf x$, the \emph{generalized frieze variety} $X(Q,a_\ast)$ is defined to be the Zariski closure in $\CC^n$ of the {$\mu_\ast$-orbit} of $a_\ast$.
\item Let $\widetilde X(Q,a_\ast)$ be $X(Q,a_\ast)$ with zero dimensional components removed. So, $\widetilde X(Q,a_\ast)$ is empty when $X(Q,a_\ast)$ is finite, e.g. when $Q$ has finite type. (Theorem \ref{thm main}, Remark \ref{rem: X(Q,a) could be finite for Q tame}.)
\end{enumerate}
\end{defn}


\begin{rem}
\begin{enumerate}
\item  By the well-known Laurent Phenomenon proved by  Fomin and Zelevinsky in \cite{FZ}, the coordinates of $\mu_\ast^t(a_\ast)$ for any integer $t$ are given by Laurent polynomials in $a_1,\cdots,a_n$. Therefore, $\mu_\ast^t(a_\ast)$ is defined for all $t$ as long as $a_\ast\in(\CC^\times)^n$, i.e. $a_i\neq0$ for all $i$.

 Moreover, by the positivity theorem  proved in \cite{LS}, if $a_\ast$ is a  positive real vector, then all $\mu_\ast^t(a_\ast)$ are   positive real vectors. In particular, every positive real vector is a generic specialization.

\item The {frieze variety} of $Q$ is $X(Q)=X(Q,(1,1,\cdots,1))$.
\item We will see that all components of $X(Q,a_\ast)$ have the same dimension and, therefore, $\widetilde X(Q,a_\ast)=X(Q,a_\ast)$ when the set is infinite. (Theorem \ref{cor: dimension of components})
\end{enumerate}
\end{rem}
 

 We {will show} that the frieze variety is invariant under mutation in the following sense.
 
\begin{thm}\label{thm: generic points of X(Q) generate X(Q)}
 If $a_\ast\in\mathbb{C}^n$ is a {generic point} on the frieze variety, then $a_\ast$ is a generic specialization of $\mathbf x$ and $X(Q,a_\ast)=\widetilde X(Q)$. {More precisely, for each component $X_i$ of $X(Q)$ of dimension $\ge1$ there is a subset $U_i\subset X_i$ given as a countable intersection of open subsets $U_i^d$ so that, for any $a_\ast$ in any $U_i$ we have $X(Q,a_\ast)=\widetilde X(Q)$.}
\end{thm}

\begin{cor}
 If $a_\ast\in X(Q)$ is a generic point then $\mu_\ast(a_\ast)\in X(Q)$.
\end{cor}

We note that 
frieze varieties often have nongeneric points. See, e.g., Remark \ref{rem: X contains nongeneric points} and the end of Example \ref{eg: Kronecker}.


\section{Proof of Theorem \ref{thm: generic points of X(Q) generate X(Q)}}

We prove a more general result (Theorem \ref{thm: generic points of X(Q,b) generate X(Q,b)} below) using the following lemmas. 


\begin{lem}\label{lem1: mu sends X to X}
 If $b_\ast$ is a generic specialization of $\mathbf x$ and $a_\ast\in \widetilde X(Q,b_\ast)\cap (\CC^\times)^n$ then \[\mu_\ast (a_\ast)\in\widetilde X(Q,b_\ast).\] 
\end{lem}

\begin{proof} Since $\widetilde X(Q,b_\ast)$ contains all but finitely many elements of $X(Q,b_\ast)$, $\widetilde X(Q,b_\ast)=X(Q,\mu_\ast^t(b_\ast))$ for sufficiently large $t>0$. Replacing $b_\ast$ with $\mu_\ast^t(b_\ast)$, we may therefore assume that $\widetilde X(Q,b_\ast)=X(Q,b_\ast)$ contains $\mu_\ast^t(b_\ast)$ for all $t\ge0$.

The variety $X(Q,b_\ast)$ is given by a finite number of polynomials $f_j$. For any $a_\ast\in (\CC^\times)^n$, $\mu_\ast(a_\ast)\in \CC^n$ is well-defined and lies in $X(Q,b_\ast)$ if and only if $f_j(\mu_\ast(a_\ast))=0$ for all $j$. Since the coordinates of $\mu_\ast(x_\ast)$ are Laurent polynomials in $x_1,\cdots,x_n$, each $f_j(\mu_\ast(x_\ast))$ is also a Laurent polynomial in the $x_i$. So, there are monomials $g_j(x_\ast)$ with the property that
\begin{equation}\label{polynomials for mu(a)}
	F_j(x_\ast):=f_j(\mu_\ast(x_\ast))g_j(x_\ast)\in \CC[x_1,\cdots,x_n].
\end{equation}
The polynomials $F_j$ have the property that, for any $a_\ast\in (\CC^\times)^n$, $\mu_\ast(a_\ast)\in X(Q,b_\ast)$ if and only if $F_j(a_\ast)=0$ for all $j$. Since $\mu_\ast^t(b_\ast)\in X(Q,b_\ast)$ for all $t\ge0$, this implies that $F_j(\mu_\ast^t(b_\ast))=0$ for all $t\ge 0$. This implies that $F_j=0$ on the Zariski closure of this set of points: $X(Q,b_\ast)$. 
 Now let $a_\ast\in X(Q,b_\ast)$,  then $F_j(a_\ast)=0$ for all $j$,  and by the above argument, $\mu_\ast(a_\ast)\in X(Q,b_\ast)=\widetilde X(Q,b_\ast)$ as claimed.
\end{proof}


In order to state the main result (Theorem \ref{thm: generic points of X(Q,b) generate X(Q,b)} which will generalize Theorem \ref{thm: generic points of X(Q) generate X(Q)}), we need to consider the irreducible components of the variety $\widetilde X(Q,b_\ast)$. We will show that the the Coxeter mutation $\mu_\ast$ cyclically permutes these components. 

\begin{thm}\label{thm: mu ast cyclically permutes components of X}
Let $b_\ast$ be any generic specialization of the  cluster  $\mathbf{x}$, for example, $b_\ast=(1,1,\cdots,1)$. Choose $t_0\ge0$ so that $\mu_\ast^t(b_\ast)\in\widetilde X(Q,b_\ast)$ for all $t\ge t_0$. Then the components of $\widetilde X(Q,b_\ast)=X(Q,\mu_\ast^{t_0}(b_\ast))$ can be numbered $X_1,\cdots,X_m$ with the following properties.
\begin{enumerate}
\item For each $t\ge t_0$, $\mu_\ast^t(b_\ast)\in X_i$ if and only if $t\equiv i$ modulo $m$. In particular, for each $t\ge t_0$, $\mu_\ast^t(b_\ast)$ lies in exactly one $X_i$.
\item $X_i$ is the closure of the set of all $\mu_\ast^{i+km}(b_\ast)$ for all integers $k\ge t_0/m$.
\item 
For any $a_\ast\in X_i\cap (\CC^\times)^n$ we have $\mu_\ast(a_\ast)\in X_{i+1}$ (or $X_1$ if $i=m$).
\end{enumerate}
\end{thm}

\begin{proof}
By replacing $b_\ast$ with $\mu_\ast^{t_0}(b_\ast)$ we may assume that $t_0=0$ and $\widetilde X(Q,b_\ast)=X(Q,b_\ast)$.

Let $X_1,\cdots,X_m$ be the components of $X(Q,b_\ast)$.

Consider the polynomials $f_{jk}$ which define the component $X_j$. As in \eqref{polynomials for mu(a)}, there are polynomials $F_{jk}$ with the property that, for any $a_\ast\in (\CC^\times)^n$, $\mu_\ast(a_\ast)\in X_j$ if and only if $F_{jk}(a_\ast)=0$ for all $k$. Then the polynomials $F_{jk}$ and $f_{ip}$ define a subvariety $X_{ij}$ of $X_i$ which contains all $a_\ast\in (\CC^\times)^n\cap X_i$ so that $\mu_\ast(a_\ast)\in X_j$. By Lemma \ref{lem1: mu sends X to X}, $\mu_\ast(a_\ast)\in \widetilde X(Q,b_\ast)=\cup X_j$, for all $a_\ast\in (\CC^\times)^n\cap X_i$. Therefore, $X_i$ is the union of the subvarieties $X_{ij}$. Since $X_i$ is irreducible, $X_i=X_{ij}$ for some $j$. In fact $j$ is uniquely determined by $i$, but we don't need to verify this.

The equation $X_i=X_{ij}$ implies that, for any $a_\ast\in X_i\cap (\CC^\times)^n$, $\mu_\ast(a_\ast)\in X_j$. For each $i$, choose one such $j$. Then $\pi(i)=j$ defines a mapping of the set $\{1,2,\cdots,m\}$ to itself. Statement (3) will follow from this after we show that $\pi$ is a cyclic permutation.

\underline{Claim 1}: $\pi$ is a permutation which is transitive, i.e. $\pi$ is an $m$-cycle.

Proof:  Since $b_\ast\in X(Q,b_\ast)$, it must lie in one of the components; suppose that $b_\ast\in X_i$. Then $\mu_\ast(b_\ast)\in X_{\pi(i)}$ and $\mu_\ast^t(b_\ast)\in X_{\pi^t(i)}$. Thus, the $\mu_\ast$-orbit of $b_\ast$ lies in the union of all $X_j$ where $j$ is in the $\pi$-orbit of $i$. But the closure of the $\mu_\ast$-orbit of $b_\ast$ is the union of all the $X_j$. So $\pi$ must be transitive, which also implies $\pi$ is a permutation  and, in particular, an $m$-cycle.

Since $\pi$ is an $m$-cycle, we may reindex the sets $X_i$ so that $\pi(i)=i+1$ for $i<m$ and $\pi(m)=1$ and so that $b_\ast\in X_m$.

\underline{Claim 2}: For each $t\ge0$, $\mu_\ast^t(b_\ast)$ lies in only one $X_i$: the one where $i\equiv t$ mod $m$.

Proof: Suppose not. Then $\mu_\ast^t(b_\ast)\in X_i\cap X_j\subset X_j$ where $j\neq i$.  By the division algorithm, any $s\ge t$ can be written as $s=p+km$ where $1\le p<m$. If $p\neq i$,  then $\mu_\ast^s(b_\ast)\in X_p$. If $p=i$ then $\mu_\ast^s(b_\ast)\in X_j$. So, the set of all $\mu_\ast^s(b_\ast)$ for all $s\ge t$ lies in the union of all $X_p$ for $p\neq i$ which is a contradiction.

These claims prove Statement (1). Statement (2) follows from the definition of $X_i$.
\end{proof}

Theorem \ref{thm: mu ast cyclically permutes components of X} can be strengthened as follows.

\begin{thm}\label{cor: dimension of components} For all generic specializations $b_\ast$ of $\bf x$, all components of $X(Q,b_\ast)$ have the same dimension. In particular, if $X(Q,b_\ast)$ is infinite, then $\widetilde X(Q,b_\ast)=X(Q,b_\ast)$ and $t_0=0$ in Theorem \ref{thm: mu ast cyclically permutes components of X} above.
\end{thm}

\begin{proof}
The Coxeter mutation $\mu_\ast$ and its inverse are given by Laurent polynomials which are rational functions whose denominators are monomials. By Theorem \ref{thm: mu ast cyclically permutes components of X} $\mu_\ast$ gives a bijection between dense subsets of $X_i,X_{i+1}$ which are disjoint from the coordinate hyperplanes. Therefore, the components $X_i$ of $\widetilde X(Q,b_\ast)$ are birationally equivalent and therefore have the same dimension since the dimension of an irreducible variety is the transcendence degree of its field of rational functions. When $X(Q,b_\ast)$ is finite, the Coxeter mutation clearly acts as a cyclic permutation of that set. So, assume $X(Q,b_\ast)$ is infinite.

To see that $\widetilde X(Q,b_\ast)=X(Q,b_\ast)$ in the infinite case, suppose not. Then there must one point $\mu_\ast^t(b_\ast)$ not in $\widetilde X(Q,b_\ast)$ so that $\mu_\ast^{t+1}(b_\ast)\in X_i$ for some $i$. Choose a regular function $f:\CC^n\to \CC$, i.e. a polynomial in $n$ variables, which is zero on $X_{i-1}$ and nonzero on the point $\mu_\ast^t(b_\ast)$. Composing with the rational morphism $\mu_\ast^{-1}$ gives a rational function on $X_i$ whose denominator is a monomial and whose numerator is a polynomial function $g$. Moreover, $g$ is zero on a dense subset of $X_i$ by Theorem \ref{thm: mu ast cyclically permutes components of X}(2) and thus zero on all of $X_i$, but $g$ is nonzero on $\mu_\ast^{t+1}(b_\ast)\in X_i$, since $g(\mu_\ast^{t+1}(b_\ast))=f(\mu_\ast^{t}(b_\ast))\neq0$. This contradiction shows that $\mu_\ast^t(b_\ast)\in X_{i-1}$ as claimed.
\end{proof}

\begin{lem}\label{lem: Uid}
For every component $X_i$ of $X(Q,b_\ast)$ and every integer $d>0$, there is a $p_d>d$ and a dense open subset $U_i^d$ of $X_i$ so that, for every $a_\ast\in U_i^d$, we have the following.
\begin{enumerate}
\item For every $0\le t\le p_d$, the coordinates of $\mu_\ast^t(a_\ast)$ are nonzero.
\item Any polynomial of degree $\le d$ which is zero on $\mu_\ast^t(a_\ast)$ for all $0\le t\le p_d$ will also be zero on $X(Q,b_\ast)$.
\end{enumerate}
\end{lem}

Before we prove this lemma, we will show that it implies the following generalization of Theorem \ref{thm: generic points of X(Q) generate X(Q)}. We use the notation $U_i^\infty$ for the intersection of $U_i^d$ over all $d>0$.


\begin{thm}\label{thm: generic points of X(Q,b) generate X(Q,b)}
Let $b_\ast\in\CC^n$ be any generic specialization of the cluster $\mathbf x$ of $Q$. Then, for generic $a_\ast\in X(Q,b_\ast)$, i.e.   for $a_\ast\in \bigcup U_i^\infty$, we have $X(Q,a_\ast)= X(Q,b_\ast)$.\end{thm}

\begin{proof} By (1) in Lemma \ref{lem: Uid}, every element of $U_i^\infty$ is a generic specialization of $\mathbf x$. By Lemma \ref{lem1: mu sends X to X}, the entire $\mu_\ast$-orbit of $a_\ast$ lies in $\widetilde X(Q,b_\ast)=X(Q,b_\ast)$. So, $X(Q,a_\ast)\subset X(Q,b_\ast)$. 

If $X(Q,a_\ast)\neq  X(Q,b_\ast)$ there must be a polynomial $f$ which is zero on the $\mu_\ast$-orbit of $a_\ast$ but nonzero on $ X(Q,b_\ast)$. Let $d=\deg f$. Given that $f$ is zero on the $\mu_\ast$-orbit of $a_\ast\in U_i^\infty\subset U_i^d$, we conclude by (2) in Lemma \ref{lem: Uid} that $f$ is zero on $ X(Q,b_\ast)$. This contradiction proves the theorem.
\end{proof}

\begin{rem}
 Theorem \ref{thm: generic points of X(Q) generate X(Q)} follows from Theorem~\ref{thm: generic points of X(Q,b) generate X(Q,b)} by choosing $b_\ast=(1,1,\ldots,1)$.
\end{rem}


It remains to prove the lemma:

\begin{proof}[Proof of Lemma \ref{lem: Uid}] We consider only the nontrivial case when $X(Q,b_\ast)$ is infinite.
By Theorems~\ref{thm: mu ast cyclically permutes components of X} and \ref{cor: dimension of components}, $X(Q,b_\ast)=X_0\cup \cdots\cup X_{m-1}$ and $t_0=0$ in Theorem~\ref{thm: mu ast cyclically permutes components of X}.
By (1) in Theorem \ref{thm: mu ast cyclically permutes components of X} we have that $\mu_\ast^i(b_\ast)\in X_i$. Since $\widetilde X(Q,b_\ast)=X(Q,b_\ast)$, a polynomial $f$ will be zero on $X(Q,b_\ast)$ if and only if $f(\mu_\ast^t(b_\ast))=0$ for all $t\ge i$. The key point of the proof is to show that, for $f$ of degree $\le d$, we only need to check this condition for $t\le p_d$ for some fixed $p_d>0$. This is a linear condition on the coefficients of $f$. Since the rank of a linear system is a lower semi-continuous function, there will be an open subset $U_i^d$ of $X_i$ on which this system is defined (Condition (1)) and has maximum rank. This will be the desired set.

We now construct the linear system. With $n,d$ fixed, consider the polynomial mapping
\[
	P_d: \CC^n\to \CC^{\binom{n+d}n}
\]
which sends $x_\ast=(x_1,\cdots,x_n)\in\CC^n$ to the sequence of all monomials in $x_j$ of degree $\le d$. For example, when $n=2,d=3$, we have:
\[
	P_3(x,y)=(1,x,y,x^2,xy,y^2,x^3,x^2y,xy^2,y^3).
\]
Then any polynomial function $f$ on $\CC^n$ of degree $\le d$ is given as the composition of $P_d$ with a linear mapping $f_\ast:\CC^{\binom{n+d}n}\to \CC$.

Let $B_d\subset \CC^{\binom{n+d}n}$ be the vector space span of all $P_d(\mu_\ast^t(b_\ast))$ for all $t\ge m$. Then a polynomial $f$ of degree $\le d$ is zero on $X(Q,\mu_\ast^m(b_\ast))=X(Q,b_\ast)$ if and only if $f_\ast(B_d)=0$. So, $P_d(a_\ast)\in B_d$ for all $a_\ast\in X(Q,b_\ast)$. Let $k$ be the dimension of $B_d$. Then $B_d$ has a basis consisting of $P_d(\mu_\ast^{t_j}(b_\ast))$ for some $m\le t_1<t_2<\cdots<t_k$. These vectors form a $\binom{n+d}n\times k$ matrix of rank $k$. So, there is some $k\times k$ minor $M$ of this matrix which is nonzero. Since $t_j\ge m\ge i$,  the entries of the matrix, being monomials in the coordinates of $\mu_\ast^{t_j-i}(\mu_\ast^i(b_\ast))$ for some $j$, are given as Laurent polynomials in the coordinates of $\mu_\ast^i(b_\ast)\in X_i$. Therefore, for each $i$,  the minor $M$ is a Laurent polynomial in the coordinates of $\mu_\ast^i(b_\ast)$. Let $F_i$ be the numerator of this polynomial. Then $F_i(\mu_\ast^i(b_\ast))\neq 0$.

Let $p_d=t_k-i$ and let $V_i^d$ be the subset of $X_i$ of all points $a_\ast$ so that $\mu_\ast^t(a_\ast)$ is defined with nonzero coordinates for $0\le t\le p_d$. Since this is an open condition and $\mu_\ast^i(b_\ast)\in V_i^d$, $V_i^d$ is a dense open subset of $X_i$. By Lemma \ref{lem1: mu sends X to X}, $\mu_\ast^t(a_\ast)\in X(Q,b_\ast)$ for all $0\le t\le p_d$. Therefore, $P_d(\mu_\ast^t(a_\ast))\in B_d$ for all $0\le t\le p_d$. The condition that the vectors $P_d(\mu_\ast^t(a_\ast))$, for $0\le t\le p_d$ span $B_d$ is an open condition which holds for $a_\ast=\mu_\ast^i(b_\ast)\in X_i\cap V_i^d$. Therefore, it holds on some Zariski open neighborhood of $\mu_\ast^i(b_\ast)$ in $X_i\cap V_i^d$. In fact, this condition will hold on the open subset $U_i^d$ of $X_i\cap V_i^d$ on which $F_i\neq0$.

Since $U_i^d\subset V_i^d$,  then $a_\ast\in U_i^d$ will satisfy Condition (1). For any polynomial $f$ of degree $\le d$ which is zero on $\mu_\ast^t(a_\ast)$ for all $0\le t\le p_d$,  the corresponding linear map $f_\ast$ will vanish on the vector $P_d(\mu_\ast^t(a_\ast))$
for all $0\le t\le p_d$. Since these vectors span $B_d$, $f_\ast(B_d)=0$. This implies that $f$ is zero on the set $X(Q,b_\ast)$, proving Condition (2) and concluding the proof of Lemma \ref{lem: Uid}, Theorems \ref{thm: generic points of X(Q,b) generate X(Q,b)} and \ref{thm: generic points of X(Q) generate X(Q)}.
\end{proof}

We illustrate some of the concepts in the proof of Lemma \ref{lem: Uid} with two examples.

\begin{eg}
Consider the $A_2$ quiver
\[
	Q: \quad 1\longleftarrow 2.
\]
Then $X(Q)$ has only five points $\mu_\ast^i(b_\ast)=$ $(1,1), (2,3), (2,1), (1,2), (3,2)$ for $i=0,1,2,3,4$. Thus $m=5$ and each $X_i$ consists of one point. For $d=2$, $\binom{n+d}n=6$. So, the five vectors $P_2(\mu_\ast^i(b_\ast))$ do not span $\CC^6$. These five vectors are the rows of the following $5\times 6$ matrix. (The proof of Lemma \ref{lem: Uid} uses the transpose of this matrix.)
\[
\begin{array}{c|cccccc}
 & 1 & x & y & x^2 & xy & y^2\\
 \hline
 b_\ast & 1 & 1 & 1 & 1 & 1 & 1\\
 \mu_\ast(b_\ast) & 1 & 2 & 3 & 4 & 6 & 9\\
\mu_\ast^2(b_\ast) & 1 & 2 & 1 & 4 & 2 & 1\\
\mu_\ast^3(b_\ast) & 1 & 1 & 2 & 1 & 2 & 4\\
\mu_\ast^4(b_\ast) & 1 & 3 & 2 & 9 & 6 & 4
\end{array}
\]
The span of these five vectors is $B_2\subset \CC^6$. This is a hyperplane perpendicular to the vector $(3,-2,-2,1,-1,1)$. Dot product with this vector gives a linear map $f_\ast:\CC^6\to\CC$, composing with $P_2$ gives
\[
	f(x,y)=f_\ast(P_2(x,y))=x^2-xy+y^2-2x-2y+3.
\]
This is the only quadratic polynomial which vanishes on the  frieze variety $X(Q)$. The real points form an ellipse centered at $(2,2)$ with major axis going from $(1,1)$ to $(3,3)$.
\end{eg}

Here is another example which explains the minor $M$ and numerator $F$.

\begin{eg}\label{eg: Kronecker}
Consider the Kronecker quiver
\[
	Q:\quad 1\Longleftarrow 2
\]
Consider the  frieze variety $X(Q)$. The first three points are $b_\ast=(1,1)$, $\mu_\ast(b_\ast)=(2,5)$, $\mu_\ast^2(b_\ast)=(13, 34)$.

Take $d=1$. Then $\binom{n+d}n=3$. In order to span $\CC^3$ we need three vectors: $P_1(\mu_\ast^t(b_\ast))$ for $t=0,1,2$. These are the rows of the following matrix.
\[
\begin{array}{c|ccc}
 & 1 & x & y \\
 \hline
 b_\ast & 1 & 1 & 1 \\
 \mu_\ast(b_\ast) & 1 & 2 & 5 \\
\mu_\ast^2(b_\ast) & 1 & 13 & 34 \\
\end{array}
\]
Since this has full rank, the determinant of this matrix (which is $-15$) is the maximal minor. However we need the minor as a Laurent polynomial in the coordinates of $\mu_\ast^i(b_\ast)$. Take $i=1$ and write $\mu_\ast^1(b_\ast)=(y_1,y_2)$. Thus $y_1=x_1', y_2=x_2'$. In terms of the cluster $y_1,y_2$, the $3\times 3$ matrix under consideration is:
\[
\left[\begin{matrix}
1 & \frac{y_1^4+y_2^2+2y_1^2+1}{y_1y_2^2}&  \frac{y_1^2+1}{y_2}\\
1 & y_1 & y_2\\
1 &\frac{y_2^2+1}{y_1} & \frac{y_1^4+2y_2^2+y_1^2+1}{y_1^1y_2}
\end{matrix}
\right]
\]
The determinant of this matrix is the rational function $M$. The numerator of $M$ is the polynomial $F=y_1^2y_2^2M$. This is a polynomial of degree 8 in $y_1,y_2$. The reason we use these variables is because we are looking for points $a_\ast=(a_1,a_2)$ close to $\mu_\ast^1(b_\ast)=(2,5)$. What we have already calculated is: $M(2,5)=-15$.

Using invariant rational functions, the generalized  frieze variety $X(Q,b_\ast)$ can be given as follows. The rational function
\[
	h(\mathbf x)=\frac{x_1^2+x_2^2+1}{x_1x_2}
\]
is equal, as an element of $\CC(\mathbf x)$, on all iterated Coxeter mutations of the cluster $\mathbf x$. To see this, write it as:
\[
	h(x_1,x_2)=\frac{x_1+x_1'}{x_2}=h(x_1',x_2)
\]
which is invariant under Coxeter mutation since $\mu_1$ switches the terms $x_1,x_1'$ and similarly for $\mu_2$. So, it is invariant under $\mu_2\circ \mu_1$.
 At $b_\ast=(1,1)$ it takes the value $h(1,1)=3$. This makes
\begin{equation}\label{eq for Kronecker}
	x_1^2+x_2^2+1=3x_1x_2
\end{equation}
at all points in $X(Q)$. \ {Note that equation \eqref{eq for Kronecker} is a specialization of the Markov equation $x_1^2+x_2^2+x_3^2=3x_1x_2x_3.$}

Using this equation we can see that the value of $M$ at any point $(y_1,y_2)$ in $X(Q)$ is equal to $-15$. So, $F=-15y_1^2y_2^2$. The set $U_1^d$ for $d=1$ is given by $F(\mu_\ast^t(a_\ast))\neq 0$ for three values of $t$, namely $t=-1,0,1$ since we are thinking of $a_\ast$ as a specialization of $(y_1,y_2)=(x_1',x_2')$. This makes $U_1^d=V_1^d$ the complement in $X(Q)$ of the 12 points consisting of $\mu_\ast^{-t}$, for $t=-1,0,1$, applied to the points $(0,\pm \sqrt{-1}),(\pm\sqrt{-1},0)$.
\end{eg}
 
 \begin{rem}\label{rem: X(Q,a) could be finite for Q tame}
We observe that Theorem \ref{thm main} does not always hold for the generalized frieze variety. For example, when $Q$ is the Kronecker quiver considered above and $b_\ast=\left(
\frac{\sqrt2 i}{2}, \sqrt 2 i
\right)$, $X(Q,b_\ast)$ consists of only two points, $b_\ast$ and $-b_\ast$. However, we believe that, for almost all $b_\ast$, the analogue of Theorem \ref{thm main} hold.
\end{rem}


\section{Construction from invariant rational functions}

For any $k\ge 0$, the coordinates of $\mu_\ast^k(\mathbf x)$ are Laurent polynomials in $\mathbf x$. Furthermore,  each coordinate of $\mathbf x$ is given as a Laurent polynomial in $\mu_\ast^k(\mathbf x)$. So, the set of values of $\mu_\ast^k(\mathbf x)$ is not contained in any hypersurface in $\CC^n$. So, for any rational function $h(\mathbf x)\in \CC(\mathbf x)$ and any $t\ge0$, we have another rational function
$h(\mu_\ast^t(\mathbf x))\in \CC(\mathbf x)$ since the denominator of $h(\mu_\ast^t(\mathbf x))$ cannot be identically zero. Suppose, furthermore, that $h(\mathbf x)$ is a Laurent polynomial in $\mathbf x$ and $a_\ast\in\CC^n$ is a generic specialization of $\mathbf x$. Then $h(\mu_\ast^t(a_\ast))$ is a well-defined complex number for any $t\ge0$. This is particularly useful when $h(\mathbf x)$ is periodic in the following sense.

\begin{defn}\label{def: invariance under mu-k}
We say that a rational function $h(\mathbf x)$ is \emph{invariant under $\mu_\ast^k$} if:
\begin{equation}\label{mutation invariant rat func}
	h(\mu_\ast^k(\mathbf x))=h(\mathbf x)
\end{equation}
as an element of $\CC(\mathbf x)$. If $k>0$ is minimal and $h(\mathbf x)$ is Laurent, we say that $h(\mathbf x)$ is an \emph{invariant Laurent polynomial for $Q$ of period $k$}. For each $t\ge0$ we will use the notation:
\begin{equation}\label{eq: ft and gt}
	h(\mu_\ast^t(\mathbf x))=\frac{f_t(\mathbf x)}{g_t(\mathbf x)}
\end{equation}
Note that $f_t,g_t\in \CC[\mathbf x]$ depend only on the residue class of $t$ modulo the period $k$.
\end{defn}

\begin{prop}\label{prop: formulas from one h}
Let $h(\mathbf x)$ be an invariant Laurent polynomial of period $k$. Let $a_\ast$ be a generic specialization of $\mathbf x$. For each $t\ge0$, let $c_t= h(\mu_\ast^t(a_\ast))$ and let $f_t(\mathbf x), g_t(\mathbf x)$ be as in \eqref{eq: ft and gt}. 
 For $0\le t<k$, let
\[
	F_{j,t}(\mathbf x):=f_t(\mathbf x)-c_{t+j}g_t(\mathbf x)
\]
be the numerator of the rational function $h(\mu_\ast^t(\mathbf x))-h(\mu_\ast^{j+t}(a_\ast))$, and let 
$X_j$ be the intersection of the $k$ hypersurfaces given by $F_{j,t}(\mathbf x)=0$, for $0\le j<k$. 

Then the generalized quiver variety $X(Q,a_\ast)$ is contained in the union $X_0\cup X_1\cup\cdots\cup X_{k-1}$. 
\end{prop}

\begin{proof} For any $0\le j<k$ and $s\ge0$, let $b_\ast=\mu_\ast^{j+ks}(a_\ast)$. Then, for any $t\ge0$, we have $h(\mu_\ast^t(b_\ast))=h(\mu_\ast^{j+t}(a_\ast))=c_{j+t}$. 
 Since $F_{j,t}(\mathbf{x})$ is the numerator of $h(\mu_\ast^t(\mathbf x))-h(\mu_\ast^{j+t}(a_\ast))$, we get $F_{j,t}(b_\ast)=0$.
Therefore, $b_\ast=\mu_\ast^{j+ks}(a_\ast)$ lies in $X_j$ for all $s\ge0$ and the union of the $X_j$ contains the entire $\mu_\ast$-orbit of $a_\ast$. So, $X(Q,a_\ast)\subset \bigcup X_j$.
\end{proof}

\begin{rem} Thus, the single Laurent polynomial $h(\mathbf x)$ generates $k^2$ polynomials $F_{j,t}(\mathbf x)$ giving $k$ varieties $X_j$ whose union contains $X(Q,a_\ast)$ and, in many cases, is equal to $X(Q,a_\ast)$ as shown in several examples below.
In these examples, all of the rational functions $h(\mu_\ast^t(\mathbf x))=f_t(\mathbf x)/g_t(\mathbf x)$ are Laurent polynomials with positive integer coefficients. This is reflected in the fact that the monomials in the polynomials $F_{j,t}(\mathbf x)=f_t(\mathbf x)-c_{j+t}g_t(\mathbf x)$ have the same sign except for one: $-c_{j+t}g_t(\mathbf x)$. We note that this is not a general phenomenon since, e.g., when $Q$ has finite type, $\mu_\ast^k(\mathbf x)=\mathbf x$ for some $k$ and, therefore, every Laurent polynomial will be invariant with period dividing $k$.
\end{rem}

\section{Examples} 
 To illustrate  Proposition \ref{prop: formulas from one h}, we give two examples, both tame, where  a single invariant Laurent polynomial $h(\mathbf x)$ whose period $k$ is one less than the number of vertices of $Q$ gives the complete decomposition of $X(Q)$ as a union of $k$ curves.

\subsection{The affine quiver $\tilde A_2$}
Let $Q$ be the quiver:
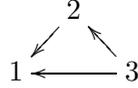
\[
\xymatrixrowsep{10pt}\xymatrixcolsep{10pt}
\xymatrix{
& 2\ar[dl]\\
1& &3\ar[ll]\ar[lu]
	}
\]
This quiver has the property that $\mu_1Q\cong Q$ after renumbering the vertices. In terms of the cluster variables $(x,y,z)$, after one mutation, we get back the same quiver with new variables $(y,z,x')$ where $x'=\frac1x(yz+1)$. For any rational function $h(x,y,z)$ let \[
h'(x,y,z):=h(y,z,x'),\]
{where $x'$ is the cluster variable obtained from the cluster $(x,y,z)$ by mutation in $x$. {
We also use the notation $h'=h\circ \tilde\mu$ where
\begin{equation}\label{eq: tilde mu}
	\tilde\mu(x,y,z):=(y,z,x').
\end{equation}
}
Note that $h^{(3)}(x,y,z)=h'''(x,y,z)=h(\mu_\ast(x,y,z))$.}

For example, let $h(\mathbf x)$ be the Laurent polynomial
\[
	h(x,y,z)=\frac{x+z}y
\]
Then $h',h''$ are given by
\[
	h'(x,y,z)=h(y,z,x')=\frac{y+x'}z=\frac{xy+yz+1}{xz}
\]
\[
	h''(x,y,z)=h'(y,z,x')=h(z,x',y')=\frac{z+y'}{x'}
\]
Observe that
\[
	\frac{h}{h''}=\frac{(x+z)x'}{y(z+y')}=\frac{yz+1+zx'}{yz+1+zx'}=1
\]
and thus 
 $h''=h$ and, consequently,  
 $h^{(a)}= h$ if $a$ is even, and $h^{(a)}= 
h'$ if $a$ is odd. 
Therefore $h$ and $h'$ are invariant under $\mu_\ast^2$:
\[
	h(\mu_\ast^2(x,y,z))=h^{(6)}(x,y,z)=h(x,y,z),
\]
and similarly, {$h'\circ \mu_\ast^2=h^{(7)}=h'$}.

Thus, $h(\mathbf x)$ is an invariant Laurent polynomial for $Q$ of period $2$. So, Proposition \ref{prop: formulas from one h} applies with 
\[
	c_0=h(1,1,1)=2
\]
\[
	c_1=h'(1,1,1)=3.
\]
So, $X(Q)$ is contained in the union of two curves $X_0\cup X_1$ where $X_0$ is given by the polynomial equations $F_{00}=F_{01}=0$ where
\[
	F_{00}(x,y,z)=Num(h(x,y,z)-c_0)=x+z-2y
\]
\[
	F_{01}(x,y,z)=Num(h'(x,y,z)-c_1)=xy+yz+1-3xz
\]
and $X_1$ is given by $F_{11}=F_{12}=0$ where
\[
	F_{10}(x,y,z)=Num(h(x,y,z)-c_1)=x+z-3y
\]
\[
	F_{11}(x,y,z)=Num(h'(x,y,z)-c_0)=xy+yz+1-2xz.
\]

From these equations it is easy to verify the observation from \cite{LLMSS} that $X_0$ is a nonsingular degree 2 curve. Indeed the equation $F_{00}=0$ is equivalent to the linear equation $z=2y-x$ which reduced the second equation to $F_{01}(x,y,z)=F_{01}(x,y,2y-x)=0$ which is a nondegenerate quadratic in two variables. 
Thus $X_0$ is a nonsingular curve in $\CC^3$ containing the infinite set of points $\mu_\ast^{2k}(1,1,1)$ for $k\ge0$. So, it must be the closure of this set. Similarly, the curve $X_1$ must be the closure of the set of all $\mu_\ast^{2k+1}(1,1,1)$. We therefore see that the frieze variety $X(Q)$ has two components given by the above four polynomials. These polynomials come from an example worked out in \cite{LLMSS}, but here all four polynomials come from the same Laurent polynomial $h$.
 
 \begin{rem}\label{rem: X contains nongeneric points}
 We note that frieze varieties often have nongeneric points. For example, the first component $X_0$ of the frieze variety discussed above contains the point $(0,\sqrt{2}i/2,\sqrt{2}i)$. Mutation at the first vertex gives $x'=(yz+1)/x=0/0$ which is undefined. However, $\tilde\mu$ (defined in \eqref{eq: tilde mu} above) sends the 0-component of $X(Q)$ to the 1-component $X_1$. So, the value of $x'$ can be computed from the linear equation $F_{10}(y,z,x')=0$:
 \[
 	x'=3z-y=\frac52\sqrt2 i.
 \]
 \end{rem}

\subsection{The affine quiver $\tilde A_n$}

More generally, consider the quiver:
\[
\xymatrixrowsep{10pt}\xymatrixcolsep{10pt}
\xymatrix{
&& 1\ar[dl] & 2\ar[l] &\ \ar[l]& \cdots &\ & n-1\ar[l]\\
Q: & 0& &&& &&&n\ar[lllllll]\ar[lu]
	}
\]
for $n\ge3$. Let $h$ be the Laurent polynomial:
\[
	h(x_0,x_1,\cdots,x_n)=\frac{x_{n-2}+x_n}{x_{n-1}}.
\]
Mutation gives:
\[
h'(x_0,\cdots,x_n):=h(x_1,x_2,\cdots,x_n,x_0)=\frac{x_{n-1}+x_0'}{x_{n}}=\frac{x_{n-1}x_0+x_1x_n+1}{x_0x_n}.
\]
Mutating $k$ times for $3\le k\le n$ ($k=2$ is given in \eqref{eq: clean formula for h(k)} below) gives
\[
	h^{(k)}(x_0,x_1,\cdots,x_n)=h(x_k,x_{k+1},\cdots,x_n,x_0',\cdots,x_{k-1}')=\frac{x_{k-3}'+x_{k-1}'}{x_{k-2}'}
\]
For $k=n$, $h^{(n)}=h$ since the quotient is:

\[
	\frac{h}{h^{(n)}}=\frac{(x_{n-2}+x_n)x_{n-2}'}{x_{n-1}(x_{n-3}'+x_{n-1}')}=\frac{x_{n-1}x_{n-3}'+1+x_nx_{n-2}'}{x_{n-1}x_{n-3}'+1+x_nx_{n-2}'}=1
\]
This gives us a cleaner formula for $h^{(k)}$ for $2\le k\le n$ using $2-n\le k-n\le0$:
\begin{equation}\label{eq: clean formula for h(k)}
	h^{(k)}(\mathbf x)=h^{(k-n)}(\mathbf x)=h(x_{k+1}',\cdots,x_n',x_0,x_1,\cdots,x_k)=\frac{x_{k-2}+x_k}{x_{k-1}}.
\end{equation}
The equation $h=h^{(n)}$ also implies that  $h^{(nk)}=h$, for all $k\ge 0$. So, 
\[
	h(\mu_\ast^k(\mathbf x))=h^{k(n+1)}(\mathbf x)=h^{(k)}(\mathbf x).
\]
In particular $h(\mu_\ast^n(\mathbf x))=h^{(n)}(\mathbf x)=h(\mathbf x)$. So, $h(\mathbf x)$ is an invariant Laurent polynomial for $Q$ of period dividing $n$. To see that the period of $h$ is exactly $n$ we compute $c_k$:
\begin{equation}\label{eq: value of ck}
	c_k=h^{(k)}(1,1,\cdots,1)=h(1,1,\cdots,1,2,3,\cdots,k+1)=\begin{cases} 3 & \text{if } k=1;\\
    2& \text{if } 2\le k\le n.
    \end{cases}
\end{equation}

Also, $h^{(k)}$ is invariant under $\mu_\ast^n$, for all $k\ge0$, since $h^{(k)}(\mu_\ast^{sn}(\mathbf x))=h^{(k+sn)}(\mathbf x)=h^{(k)}(\mathbf x)$. So, for generic $a_\ast$, the frieze variety $X(Q,a_\ast)$ has $n$ components 
\[
	X(Q,a_\ast)= X_0\cup  X_1\cup\cdots \cup X_{n-1};
\]
 the first component $ X_0$, containing the $\mu_\ast^n$-orbit of $a_\ast=(1,1,\cdots,1)$, is given by the $n$ polynomials $F_{01},\cdots,F_{0,n-1},F_{0n}=F_{00}$  where 
	\[ 
	F_{01}=Num(h'(\mathbf x)-c_1)=x_0x_{n-1}+x_1x_n+1-3x_0x_n,
\]
 by Proposition \ref{prop: formulas from one h} and  \eqref{eq: value of ck}; and for $2\le k\le n$ 
\[
	F_{0k}=Num(h^{(k)}(\mathbf x)-c_k)=x_{k-2}+x_k-2x_{k-1},
\]
 by   Proposition \ref{prop: formulas from one h} ,\eqref{eq: clean formula for h(k)} and \eqref{eq: value of ck}.

\begin{rem}
Note that, for any point in $X_0$, the equations $F_{0k}=0$ for $2\le k\le n$ express the coordinates $x_2,x_3,\cdots,x_n$ as linear combinations of $x_0,x_1$. Thus, points in $X_0$ are determined by their first two coordinates.
Since $X_0$ contains the 3 points $(1,1,\cdots), (1,2,\cdots), (n+1,2n+3,\cdots)$ and their negatives (where we ignore all but the first two coordinates) and does not contain $(0,0,\cdots)$, linear algebra in $\mathbb C^2$ shows $X_0$ cannot be a union of two straight lines. So, $X_0$ is an irreducible degree 2 curve containing the $\mu_\ast^n$-orbit of the point $a_\ast=(1,1,\cdots,1)$. Since no two distinct curves can have an infinite intersection, $X_0$ is the Zariski closure of this set.
\end{rem}

By Proposition \ref{prop: formulas from one h}, $X_1$ is given by the $n$ polynomial equations $F_{1,t}=0$ for $0< t\le n$
\[
	F_{1n}=Num(h(\mathbf x)-c_1)=x_{n-2}+x_n-3x_{n-1}
\]
\[
	F_{11}=Num(h'(\mathbf x)-c_2)=x_0x_{n-1}+x_1x_n+1-2x_0x_n
\]
\[
	F_{1k}=F_{0k}, \text{ for }k=2,\cdots,n-1.
\]
The other polynomials $F_{jt}$ are similar. As in the case of $\tilde A_2$, all polynomials $F_{jk}$ are given by the single Laurent polynomial $h$ and its $n-1$ mutations $h^{(k)}$ for $1\le k<n$.

\section{Symmetry}

One easy observation {\cite{LLMSS}} is that, if a permutation $\sigma$ of $\{1,2\cdots,n\}$ leaves the quiver $Q$ invariant, then the {frieze} variety $X(Q)$ satisfies $x_k=x_{\sigma(k)}$ for all $k$. In terms of invariant rational functions, $h=x_{\sigma(k)}/x_k$ is invariant under $\mu_\ast^2$ since $\mu_\ast$ inverts $h$.

A similar result holds true if a mutation of $Q$ has symmetry. For example,
\[
	Q:\quad 1\Longleftarrow 2\Longleftarrow 3
\]
becomes symmetric after one mutation:
\[
	\mu_1(Q):\quad 1'\Longrightarrow 2\Longleftarrow 3
\]
This implies that
\[
	h=\frac{x_1'}{x_3}=\frac{x_2^2+1}{x_1x_3}
\]
is invariant under $\mu_\ast^2$. Since $c_0=h(1,1,1)=2$ and $c_1=1/c_0=1/2$, the {frieze} variety of $Q$ is $X(Q)=X_0\cup X_1$, where $X_0$ is the hypersurface given by $F_0=0$, where
\[
	F_0=Num(h(\mathbf x)-c_0)=x_2^2+1-2x_1x_3
\]
and $X_1$ containing $\mu_\ast(1,1,1)=(2,5,26)$ is the hypersurface given by $F_1=0$, where
\[
	F_1=Num\left(h(\mathbf x)-\frac12\right)=Num\left(2{h(\mathbf x)}-1\right)=2x_2^2+2-x_1x_3.
\]
\begin{rem}
The hypersurface $X_0$ contains the $\mu_\ast^2$-orbit of the point $(1,1,\cdots,1)$. Since this set is not contained in any curve by Theorem \ref{thm main}, there cannot be a smaller variety containing this set. So, $X_0$ is the Zariski closure of the $\mu_\ast^2$ orbit of $(1,1,\cdots,1)$ and similarly for $X_1$.
\end{rem}
More generally we have the following.

\begin{prop}\label{prop: symmetry of a mutation of Q}
Suppose that $i$ is a sink in the quiver $Q$ and $j$ is a source so that, for any other vertex $k$, the number $n_k$ of arrows from $k$ to $i$ is equal to the number of arrows from $j$ to $k$. Then, the  frieze variety $X(Q)$ is contained in the union of two hypersurfaces $X_0,X_1$ given by the equations $F_0=0$ and $F_1=0$, where
\[
	F_0=1-2x_ix_j+\prod_k x_k^{n_k}
\]
\[
	F_1=2-x_ix_j+2\prod_k x_k^{n_k}.
\]
\end{prop}

\begin{proof}
After the mutation $\mu_\ast=\mu_n\circ\cdots\circ\mu_1$, we will have $x_\ast'$ where
\[
	x_i'=\frac{1+\prod x_k^{n_k}}{x_i},\quad x_j'=\frac{1+\prod (x_k')^{n_k}}{x_j}
\] 
So, the rational function $h(\mathbf x)=x_i'/x_j$ will mutate to 
\[
	h(\mu_\ast(\mathbf x))=\frac{x_i''}{x_j'}=\left(
	\frac{1+\prod (x_k')^{n_k}}{x_i'}
	\right)\left(
	\frac{x_j}{1+\prod (x_k')^{n_k}}
	\right)=\frac{x_j}{x_i'}=\frac1{h(\mathbf x)}.
\]
So, $h(\mathbf x)$ is $\mu_\ast^2$ invariant.
Since $h(\mathbf x)=f(\mathbf x)/g(\mathbf x)$ where $f(\mathbf x)=1+\prod x_k^{n_k}$ and $g(\mathbf x)=x_ix_j$, $c_0=h(1,1,\cdots,1)=2$, $c_1=1/c_0=\frac12$, the numerator of $h(\mathbf x)-c_0$ is $F_0$ and the numerator of$h(\mathbf x)-c_1$ is $F_1$. By Proposition \ref{prop: formulas from one h}, the $\mu_\ast^2$ orbit of $(1,1,\cdots,1)$ satisfies $F_0=0$ and the $\mu_\ast^2$ orbit of $\mu_\ast(1,1,\cdots,1)$ satisfies $F_1=0$.
\end{proof}

\section{Questions and answers}\label{sec: Q+A}

We list a few questions from the first version of this paper and short answers to these questions following suggestions by Gordana Todorov. Details will be given in another paper.

\begin{enumerate}
\item  Question: In the tame case, does an invariant Laurent polynomial $h(\mathbf x)$ always exist? 
Answer: Yes. The cluster character of a regular module in a tube or rank $k$ will be an invariant Laurent polynomial of period $k$.
\item Question: Can  $h(\mathbf {x})$ be chosen to have positive integer coefficients? Answer: In the tame case yes. In the wild case we also believe the answer is yes since we believe that the only invariant rational functions are the ones given by symmetry of the quiver as in Proposition \ref{prop: symmetry of a mutation of Q}.
\item  Question: Can  $h(\mathbf {x})$ be chosen such that all iterated Coxeter mutations $h(\mu_\ast^t(\mathbf x))$ have positive integer coefficients? Answer: In the tame case, yes. The answer also seems to be yes in the wild case if, as we suspect, the only invariant rational functions come from symmetry of the quiver as observed in \cite{LLMSS} or as given in Proposition \ref{prop: symmetry of a mutation of Q}. However we note that, in the latter case, $h(\mu_\ast(\mathbf x))=1/h(\mathbf {x})$ is not Laurent.
\item Question: Is the period of the invariant Laurent polynomial always equal to the number of components of $X(Q,a_\ast)$?
Answer:  No, a counterexample is given by the following quiver
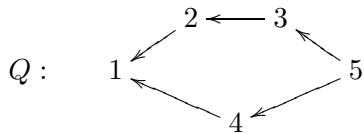
\[
\xymatrixrowsep{5pt}\xymatrixcolsep{5pt}
\xymatrix{
&&&& 2\ar[lld] && 3\ar[ll] && \\
Q:&& 1 &&&&&&5\ar[llld]\ar[llu]\\
&&&&&4\ar[lllu] &&
	}
\]
Here there is a tube of rank 3 giving an invariant Laurent polynomial of period 3 and another tube of rank 2 giving an invariant Laurent polynomial of period 2. This suggest that there should be 6 components.
\end{enumerate}


\end{document}